\documentclass{amsart}    
\usepackage{amsmath}    
\usepackage{graphicx}   
\usepackage{verbatim}   
\usepackage{amssymb}
\usepackage{amsthm}
\usepackage{mathrsfs}
\usepackage{exscale}
\newtheorem{theorem}{Theorem}[section]
\newtheorem{lemma}[theorem]{Lemma}
\newtheorem{corollary}[theorem]{Corollary}
\newtheorem{proposition}[theorem]{Proposition}

\newcommand{\dint}{
\displaystyle\int}

\newcommand{\dsum}{
\displaystyle\sum}

\newcommand{\dpiint}{
\dint_{-\pi}^\pi}

\newcommand{\dpiifrac}{
\dfrac{1}{2\pi}}

\newcommand{\inflim}[1]{
\underset{#1\to\infty}\lim}

\newcommand{\zsum}[1]{
\underset{#1\in\Z}\dsum}
\newcommand{\zsumzero}[1]{
\underset{#1\in\Z\setminus\{0\}}\dsum}

\newcommand{\R}{
\mathbb{R}}

\newcommand{\N}{
\mathbb{N}}
\newcommand{\Z}{
\mathbb{Z}}
\newcommand{\F}{
\mathscr{F}}

\newcommand{\bracket}[1]{
\left\langle#1\right\rangle}

\newcommand{\I}{
\mathscr{I}}
\newcommand{\Ih}{
\mathscr{I}_{_{h^2}}^X}

\newcommand{\rootpi}{
\dfrac{1}{\sqrt{2\pi}}}

\title[Interpolation of Sobolev Functions via Gaussians]{Approximation rates for interpolation of Sobolev functions via Gaussians and allied functions}
\author{Keaton Hamm}
\thanks{This work is part of the author's doctoral dissertation.  He thanks his advisors Th. Schlumprecht and N. Sivakumar for their guidance.  Research partially supported by National Science Foundation grant DMS 1160633}

\address{Department of Mathematics, Texas A\&M University, College Station, Texas, 77843}
\email{khamm@math.tamu.edu} 
\begin{document}

\begin{abstract}
A \textit{Riesz-basis sequence for} $L_2[-\pi,\pi]$ is a strictly increasing sequence $X:=(x_j)_{j\in\Z}$ in $\R$ such that the set of functions $\left(e^{-ix_j(\cdot)}\right)_{j\in\Z}$ is a Riesz basis for $L_2[-\pi,\pi]$.  Given such a sequence and a parameter $0<h\leq1$, we consider interpolation of functions $g\in W_2^k(\R)$ at the  set $(hx_j)_{j\in\Z}$ via translates of the Gaussian kernel.  Existence is shown of an interpolant of the form
$$I^{hX}(g)(x):=\zsum{j}a_je^{-(x-hx_j)^2},\quad x\in\R,$$
which is continuous and square-integrable on $\R$, and satisfies the interpolatory condition $I^{hX}(g)(hx_j)=g(hx_j),j\in\Z$.  Moreover, use of the parameter $h$ gives approximation rates of order $h^k$.  Namely, there is a constant independent of $g$ such that $\|I^{hX}(g)-g\|_{L_2(\R)}\leq Ch^k|g|_{W_2^k(\R)}$.  Interpolation using translates of certain functions other than the Gaussian, so-called \textit{regular interpolators}, is also considered and shown to exhibit the same approximation rates.
\end{abstract}

\maketitle
 \allowdisplaybreaks
 \section{Introduction}\label{Introduction}
 
 This paper is concerned with the interpolation of a certain class of real-valued Sobolev functions at infinite sets of scattered points via translates of Gaussian kernels.  The study of scattered-data interpolation arose naturally out of the theory of cardinal interpolation, which deals with function interpolation at the integer lattice by interpolants formed by integer shifts of a single function.  A classical example of such a study is the theory of cardinal spline interpolation where the interpolants are formed by piecewise polynomials.  This subject was pioneered by Schoenberg, and was developed to a large extent by him and others who followed.  More recently, cardinal interpolation by integer shifts of Gaussian functions was considered in a series of papers by Baxter, Riemenschneider and Sivakumar (\cite{brs}, \cite{ri}-\cite{rs3}, and \cite{siva}).  It is natural then to consider whether the types of results obtained in cardinal interpolation still hold in the case of more general data sites.  Lyubarskii and Madych \cite{lm} studied such a problem for
spline functions, and subsequently, Schlumprecht and Sivakumar \cite{schsiv} proved analogous results for Gaussian interpolation.  
 
 One of the principal motivating factors in our study was a recent paper by Hangelbroek, Madych, Narcowich, and Ward \cite{hmnw}, in which they studied $L_p$ approximation rates of sufficiently smooth Sobolev ($W_p^k$) functions for cardinal interpolation with Gaussian kernels. We give an analogue of a special case of \cite{hmnw}; namely we consider approximation rates for interpolation of $W_2^k$ functions at nonuniformly spaced points.  This gives us a generalization of the results in \cite{schsiv} to a broader class of functions, as well as giving approximation rates in a more general setting.  The reason that we do not have results for general $p\neq2$ comes from the fact that the results in \cite{bss}, \cite{lm} and \cite{schsiv} are primarily about $L_2$ convergence, and to the author's best knowledge few comparable $L_p$ convergence results are known.

 Interestingly, the theory of scattered-data interpolation of bandlimited functions was recently generalized in a different direction by Ledford \cite{ledford}.  He gave sufficient conditions on a function $\phi$ which allow nonuniform sampling of bandlimited functions via shifts of $\phi$, and moreover proved convergence results in the vein of \cite{schsiv} for so-called families of \textit{regular interpolators}.  While we primarily concern ourselves with presenting the proofs of our results for the more concrete Gaussian case, we also remark that we obtain the same approximation rates in the case of regular interpolators.
 
 The rest of this paper is laid out in the following manner.  In Section \ref{SECbasic} we introduce some basic notations and techniques required for the rest of the paper. Section \ref{SECmainresults} contains an explanation of the main techniques used, the statements of the main results, and the proof of the main theorem. Sections \ref{SECinterpbybandlimited} through \ref{SECappendix} provide proofs of the theorems in Section \ref{SECmainresults} which are used in the proof of the main theorem.  In section \ref{SECledford} we display the analogous theorems for the case of regular interpolators, and finally in Section \ref{SECmultivariate} we comment on a simple multivariate extension.

\section{Basic Notation and Facts}\label{SECbasic}

If $\Omega\subset\R$ is an interval, then let $L_p(\Omega)$, $1\leq p\leq\infty$, be the usual Lebesgue space over $\Omega$ with its usual norm.  If no set is specified, we mean $L_p(\R)$.  Similarly, denote by $\ell_p(I)$ the usual sequence spaces indexed by the set $I$; if no index set is given, we refer to $\ell_p(\Z)$.  Given a function $g\in L_1$, its Fourier transform, $\widehat{g}$, is defined by the following:
\begin{equation}\label{ft}
 \widehat{g}(\xi):=\int_\R g(x)e^{-i\xi x}dx,\quad \xi\in\R.
\end{equation}
The Fourier transform can be uniquely extended to a linear isomorphism of $L_2$ onto itself.
Denote by $\F[g]$, the Fourier transform of a function $g\in L_2$.  Moreover, Parseval's Identity states that
\begin{equation}\label{parseval}
\|\F[g]\|_{L_2} = \sqrt{2\pi}\|g\|_{L_2} .
\end{equation}
If $g$ is also continuous, and $\F[g]\in L_1$, then the following inversion formula holds:
\begin{equation}
 g(x)=\dpiifrac\int_\R\F[g](\xi)e^{i\xi x}d\xi,\quad x\in\R.
\end{equation}

Additionally, define $W_2^k:=W_2^k(\R)$ to be the Sobolev space over $\R$ of functions in $L_2$ whose first $k$ weak derivatives are in $L_2$.  
The seminorm on $W_2^k$ is defined as
$$|g|_{W_2^k}:=\left(\int_\R |g^{(k)}(x)|^2dx\right)^\frac{1}{2} = \|g^{(k)}\|_{L_2},$$
and the norm on $W_2^k$ can be defined by
$$ \|g\|_{W_2^k}:=\left(\|g\|^2_{L_2}+|g|^2_{W_2^k}\right)^\frac{1}{2}.$$

Another basic property of the Fourier transform which will be used frequently is that for $g\in W_2^k$,
\begin{equation}\label{fourierderivative}
 \F[g^{(k)}](\xi) = (i\xi)^k\F[g](\xi).
\end{equation}

Hence we also have an equivalent seminorm for $W_2^k$:
\begin{equation}\label{fourierseminorm}
 |g|_{W_2^k} = \rootpi\left(\int_\R|\xi|^{2k}|\F[g](\xi)|^2d\xi\right)^\frac{1}{2} .
\end{equation}
This is the seminorm we will use throughout, since most calculations will be carried out in the Fourier transform domain.

In light of \eqref{parseval} and \eqref{fourierderivative}, we may formulate equivalent norms on the spaces $L_2$ and $W_2^k$ as follows:
$$\|g\|_{L_2}=\rootpi\left(\int_\R\left|\F[g](\xi)\right|^2d\xi\right)^\frac{1}{2}\;,\;\; \|g\|_{W_2^k}=\rootpi\left(\int_\R\left(1+|\xi|^k\right)^2\left|\F[g](\xi)\right|^2d\xi\right)^\frac{1}{2}.$$

An important class of functions for us will be the Paley-Wiener, or bandlimited, functions.  For $\sigma>0$, we define these spaces as follows:
\begin{equation}\label{DEFPW}
 PW_\sigma:=\left\{f\in L_2:\F[f]=0 \textnormal{ almost everywhere outside } [-\sigma,\sigma]\right\}.
\end{equation}
The Paley-Wiener Theorem asserts that an equivalent definition of this space is that of entire functions of exponential type $\sigma$ whose restriction to $\R$ is in $L_2$. We also define the class $PW_\sigma^k$ to be that of functions whose $k$-th derivatives lie in $PW_\sigma$, and note that this class is similar to a Beppo-Levi space in the sense that only the $k$-th derivative is required to be in $PW_\sigma$.  For example, polynomials of degree at most $k-1$ are elements of $PW_\sigma^k$ for any $\sigma$.

A sequence of functions $(\phi_j)_{j\in\Z}$ is defined to be a Riesz basis for a Hilbert space $\mathcal{H}$ if every element $h\in\mathcal{H}$ has a unique representation
\begin{equation}\label{rbasis}
 h=\zsum{j}a_j\phi_j,\quad \zsum{j}|a_j|^2<\infty,
\end{equation}
and consequently by the Uniform Boundedness Principle, there exists a constant $B$, called the basis constant, such that
\begin{equation}\label{rbasisconstant}
 \dfrac{1}{B}\left(\zsum{j}|c_j|^2\right)^\frac{1}{2}\leq\left\|\zsum{j}c_j\phi_j\right\|_{\mathcal{H}}\leq B\left(\zsum{j}|c_j|^2\right)^\frac{1}{2},
\end{equation}
for every sequence $(c_j)\in\ell_2$.  
We will consider interpolation at a given sequence of points, $X:=(x_j)_{j\in\Z}$ in $\R$, that is a Riesz-basis sequence for $L_2[-\sigma,\sigma]$.  Recall (\cite{lm}) that $X$ is a Riesz-basis sequence for $L_2[-\sigma,\sigma]$ if $x_j<x_{j+1}$ for every integer $j$, and the functions $\left(e^{-ix_j(\cdot)}\right)_{j\in\Z}$ form a Riesz basis for $L_2[-\sigma,\sigma]$.  

It is noted in \cite{lm} that a necessary condition for a sequence $X$ to be a Riesz-basis sequence for $L_2[-\pi,\pi]$ is that it be minimally and maximally separated, i.e. there exist numbers $Q\geq q>0$ such that
\begin{equation}\label{separationcondition}
 q\leq x_{j+1}-x_j\leq Q,\quad j\in\Z.
\end{equation}
A now classical sufficient condition is Kadec's so-called 1/4-Theorem (see \cite{ka}), which states that if $(x_j)_{j\in\Z}$ is a strictly increasing sequence such that
\begin{equation}\label{kadecscondition}
 \underset{j\in\Z}\sup\;|x_j-j|<1/4,
\end{equation}
then $(x_j)$ is a Riesz-basis sequence for $L_2[-\pi,\pi]$.
For more recent results concerning extensions of Kadec's theorem, see \cite{Bailey1}, \cite{Bailey2}, and \cite{chinese}.  Furthermore, necessary and sufficient conditions for $X$ to be a Riesz-basis sequence for $L_2[-\pi,\pi]$ are given by Pavlov in \cite{pav} in terms of zeroes of entire sine-type functions.

Imposing the Riesz-basis sequence condition on $X$ is natural because it allows for uniqueness.  Specifically, if $X:=(x_j)_{j\in\Z}$ is a Riesz-basis sequence for $L_2[-\pi,\pi]$ and if $f,g\in PW_\pi$ are such that $f(x_j)=g(x_j)$ for all $j$, then it follows that $f=g$.  Therefore, the interpolation operators we define in the sequel will be uniquely determined on bandlimited functions.

\section{Main Results}\label{SECmainresults}

In this section, we will formulate a series of theorems on our way to the main result, and using these we provide the proof of the main theorem of this paper, Theorem \ref{maintheorem}.

We will form two interpolants involving translates of a fixed Gaussian function; one is the original interpolation operator from \cite{schsiv} which interpolates functions in $PW_\pi$ at a Riesz-basis sequence $X$, and the second is similar to that of \cite{hmnw}, which we will use to interpolate $W_2^k$ functions at $hX$, given a parameter $0<h\leq1$.  Let $\lambda>0$ be fixed, and let $X:=(x_j)_{j\in\Z}$ be a Riesz-basis sequence.  It was shown in \cite{schsiv} that given $f\in PW_\pi$, there exists a unique $\ell_2$ sequence $(a_j)_{j\in\Z}$ depending on $\lambda, f$, and $X$ such that the Gaussian interpolant
\begin{equation}\label{DEFscriptI}
 \mathscr{I}_{\lambda}^{X}(f)(x):=\zsum{j}a_je^{-\lambda(x-x_j)^2},\quad x\in\R,
\end{equation}
is continuous and square-integrable on $\R$, and satisfies
\begin{equation}\label{ssinterpcondition}
 \mathscr{I}_\lambda^X(f)(x_j)=f(x_j),\quad j\in\Z.
\end{equation}
Where the sequence is clear, we will omit the superscript $X$.  The main result from their paper was that for a fixed Riesz-basis sequence $X$, and any $f\in PW_\pi$, $\mathscr{I}_\lambda^X(f)\to f$ in $L_2$ and uniformly on $\R$ as $\lambda\to0^+$.

We define the second interpolant via the following identity:
\begin{equation}\label{DEFIhX}
 I^{hX}(f)(x):=\frac{1}{h}\Ih\left(f^h\right)\left(\frac{x}{h}\right),
\end{equation}
where
\begin{equation}\label{DEFfh}
 f^h(x):=hf(hx),\quad x\in\R.
\end{equation}
Defined this way, we see that this operator satisfies a similar interpolation condition to \eqref{ssinterpcondition}.  Precisely, due to \eqref{ssinterpcondition}, \eqref{DEFIhX}, and \eqref{DEFfh},
\begin{equation}
 I^{hX}(f)(hx_j)=\frac{1}{h}\mathscr{I}_{_{h^2}}^X\left(f^h\right)\left(\frac{hx_j}{h}\right)=\frac{1}{h}f^h(x_j)=f(hx_j).
\end{equation}
An important fact, and indeed the reason for defining $I^{hX}$ the way we have, is that if $f\in PW_\frac{\pi}{h}$, then $f^h\in PW_\pi$. Moreover, the relation $\F[f^h](\xi)=\F[f](\xi/h)$ holds.
The second interpolation operator enters the analysis in the following way: to interpolate Sobolev functions, we first interpolate them by bandlimited functions whose band size increases depending on $h$, and then use the original Gaussian interpolant to do the rest of the work. More precisely, we have the following result:

\begin{theorem}\label{TheoreminterpsobolevbyPW}
 Let $k\in\N$, $h>0$, and let $X$ be a fixed Riesz-basis sequence for $L_2[-\pi,\pi]$.  Let $B$ be given by \eqref{rbasisconstant} for the associated Riesz basis, and let $q$ and $Q$ be as in \eqref{separationcondition}. Then for every $g\in W_2^k$, there exists a unique $F\in PW_\frac{\pi}{h}$ such that
 \begin{equation}\label{EQpwinterpcondition}
  F(hx_j)=g(hx_j),\quad j\in\Z,
 \end{equation}
\begin{equation}\label{EQbernstein}
 |F|_{W_2^k}\leq C|g|_{W_2^k},
\end{equation}
and
\begin{equation}\label{jackson}
 \|g-F\|_{L_2}\leq Ch^k|g|_{W_2^k},
\end{equation}

where $C$ is a constant depending on $B,k,q$, and $Q$.
\end{theorem}
We note that for the case $x_j=j$ but general $p$, a similar result was obtained in \cite[Lemma 2.2]{hmnw}.

We then consider stability of the Gaussian interpolant $\mathscr{I}_\lambda$ as an operator from $W_2^k$ to itself.  In \cite{schsiv}, it was shown that $(\mathscr{I}_\lambda)_{\lambda\in(0,1]}$ is uniformly bounded as a set of operators from $PW_\pi$ to $L_2$, as well as being uniformly bounded from $PW_\pi$ to $C_0(\R)$.  We adapt the techniques of that paper to show that $(\mathscr{I}_\lambda)_{\lambda\in(0,1]}$ is uniformly bounded as a set of operators from $PW_\pi$ to $W_2^k$, which in light of Theorem \ref{TheoreminterpsobolevbyPW}, means that this family is uniformly bounded on $W_2^k$.  We summarize this in the following theorem:
\begin{theorem}\label{scrIuniformbounded}
 Let  $k\in\N$ and let $X$ be a Riesz-basis sequence for $L_2[-\pi,\pi]$. Then there exists a constant $C$ depending only on $k$ and $X$ such that for every $0<\lambda\leq 1$,
 \begin{equation}\label{w2kuniformbound}
|\mathscr{I}_\lambda(g)|_{W_2^k}\leq C|g|_{W_2^k},\quad\textnormal{for all } g\in W_2^k.
 \end{equation}
 Consequently, $(\mathscr{I}_\lambda)_{\lambda\in(0,1]}$ is uniformly bounded as a set of operators from $W_2^k$ to itself.
\end{theorem}

Combined with some calculations (see Section \ref{SECcorollaryproofs}), this theorem yields the following corollary:
\begin{corollary}\label{Ihxbounded}
 Let $k\in\N$ and let $X$ be a Riesz-basis sequence for $L_2[-\pi,\pi]$. Then there exists a constant $C$ depending only on $k$ and $X$ such that for every $0<h\leq1$,
 \begin{equation}\label{Iw2kbound}
  |I^{hX}(g)|_{W_2^k}\leq C|g|_{W_2^k},\quad\textnormal{for all } g\in W_2^k.
 \end{equation}
 Consequently, $\left(I^{hX}\right)_{h\in(0,1]}$ is uniformly bounded as a set of operators from $W_2^k$ to itself.
\end{corollary}

A combination of these results yields our main theorem:
\begin{theorem}\label{maintheorem}
 Let $k\in\N$, $0<h\leq1$, and let $X$ be a Riesz-basis sequence for $L_2[-\pi,\pi]$.  Then there exists a constant depending only on $k$ and $X$ such that for every $g\in W_2^k$,
 \begin{equation}\label{maintheoremeq}
  \|I^{hX}(g)-g\|_{L_2}\leq Ch^k|g|_{W_2^k}.
 \end{equation}
\end{theorem}
We also obtain derivative convergence:
\begin{corollary}\label{CORderivativemaintheorem}
 Let $k\geq2$, $1\leq j<k$, $0<h\leq1$, and let $X$ be a Riesz-basis sequence for $L_2[-\pi,\pi]$.  Then there exists a constant depending only on $j,k$, and $X$ such that for every $g\in W_2^k$,
 \begin{equation}\label{EQderivativemaintheorem}
  |I^{hX}(g)-g|_{W_2^j}\leq Ch^{k-j}|g|_{W_2^k}\;.
 \end{equation}
\end{corollary}

As every bandlimited function is also in $W_2^k$ for every $k$, an easy corollary of Theorem \ref{maintheorem} is the following:
\begin{corollary}\label{schwartzresult}
 Let $0<h\leq 1$, $\sigma>0$, and let $X$ be a Riesz-basis sequence for $L_2[-\pi,\pi]$.  Then for each $k\in\N$, there exists a constant depending only on $k$ and $X$ such that for every $\phi\in PW_\sigma$,
 \begin{equation}\label{schwartzeq}
  \|I^{hX}(\phi)-\phi\|_{L_2}\leq Ch^k|\phi|_{W_2^k}.
 \end{equation}
\end{corollary}
Recall that $\mathcal{S}(\R)$, the space of Schwartz functions on $\R$, is the collection of infinitely differentiable functions $\phi$ such that for every $j,k\geq0$,
$$\underset{x\in\R}\sup\left|x^j\phi^{(k)}(x)\right|<\infty\;.$$
Then we note that Corollary \ref{schwartzresult} holds for $\phi\in\mathcal{S}(\R)$ as well.

Before giving the proof of our main result, we display a theorem (that we will use several times) of Madych and Potter that gives an estimate on the norm of functions with many zeroes.

\begin{theorem}[cf. \cite{mp}, Corollary 1]\label{madychpotter}
 Suppose $k\in\N$ and $f\in W_p^k(\R)$.  Let $Z:=\{x\in\R:f(x)=0\}$ and suppose that $h:=\max\{dist(x,Z):x\in\R\}<\infty$.  Then there exists a constant $C$ independent of $f,h$ and $Z$ such that
$$|f|_{W_p^j}\leq C h^{k-j}|f|_{W_p^k}$$
for $j=0,1,\dots, k$.
\end{theorem}

We now provide a proof of Theorem \ref{maintheorem} using the results we have collected so far in this section.
\begin{proof}[Proof of Theorem \ref{maintheorem}]
 Theorem \ref{TheoreminterpsobolevbyPW} provides a function $F\in PW_\frac{\pi}{h}$ which interpolates $g$ at $(hx_j)$.  Then 
 $$\|I^{hX}(g)-g\|_{L_2}\leq\|I^{hX}(g)-F\|_{L_2}+\|F-g\|_{L_2}=:I_1+I_2 .$$
 By  \eqref{jackson}, 
 $$I_2\leq Ch^k|g|_{W_2^k}.$$
 Since $F(hx_j)=g(hx_j)$, uniqueness of the interpolation operator guarantees that $I^{hX}(F)=I^{hX}(g)$, and therefore $I_1=\|I^{hX}(F)-F\|_{L_2}$.  Applying Theorem \ref{madychpotter} to $I^{hX}(F)-F$ along with Corollary \ref{Ihxbounded} and equation \eqref{EQbernstein} yields
 $$I_1\leq Ch^k|I^{hX}(F)-F|_{W_2^k}\leq Ch^k\left(|I^{hX}(F)|_{W_2^k}+|F|_{W_2^k}\right)\leq Ch^k|F|_{W_2^k} \leq Ch^k|g|_{W_2^k}.$$
\end{proof}


\section{Interpolation of $W_2^k$ functions by Bandlimited functions}\label{SECinterpbybandlimited}
The main goal of this section is to provide the proof of Theorem \ref{TheoreminterpsobolevbyPW}, which follows from a series of intermediate steps.  We start with a proposition which will lead us to the existence and uniqueness of a bandlimited interpolant for a given Sobolev function.  The proposition is stated and proved for general $p$ rather than in the specific case $p=2$ because it is interesting in its own right and gives a sort of reverse inequality of \cite[Theorem 1]{mp}.
\begin{proposition}\label{solutioninellp}
 Let $(x_n)_{n\in\Z}\subset\R$ be a strictly increasing sequence such that $\underset{n\in\Z}\inf\;\left(x_{n+1}-x_n\right)=: q>0$.  If $g\in W_p^k(\R)$, $1\leq p\leq\infty$, $k\in\N$, then $(g(x_n))_{n\in\Z}\in\ell_p$.  Moreover, if $\frac{1}{p}+\frac{1}{p'}=1$, and $p\neq\infty$, then
 \begin{equation}\label{ellpestimate}
  \|g(x_n)\|_{\ell_p}\leq 2^\frac{1}{p'}\left(\frac{3}{2q}\right)^\frac{1}{p}\|g\|_{L_p}+2^\frac{1}{p'}\left(\frac{2q}{3}\right)^\frac{1}{p'}\|g'\|_{L_p}.
 \end{equation}
\end{proposition}

\begin{proof}
 First, in the case $p=\infty$, if $g\in W_\infty^1(\R)$, then $g$ is continuous, and the result follows.
 
 For $1\leq p<\infty$, let
 $$I_n:=\left[x_n-\frac{q}{3},x_n+\frac{q}{3}\right].$$
 These intervals are pairwise disjoint and have length $\frac{2}{3}q$.  The former condition ensures that
$$\zsum{n}\int_{I_n}|g(t)|^pdt\leq\int_\R|g(t)|^pdt<\infty.$$
As $g$ admits an absolutely continuous representative (see, for example, \cite[Theorem 7.13, p.222]{leoni}), we may assume, without loss of generality, that $g$ itself is absolutely continuous.  Therefore, we may choose $y_n\in I_n$ such that $|g(y_n)|^p=\min\{|g(x)|^p:x\in I_n\}$; then we have that
$$|g(y_n)|^p\frac{2}{3}q\leq\int_{I_n}|g(t)|^pdt,$$
and consequently, $(g(y_n))_{n\in\Z}\in\ell_p$.  Moreover,
\begin{equation}\label{subtheorem}
 \left(\zsum{n}|g(y_n)|^p\right)^\frac{1}{p}\leq \left(\frac{3}{2q}\right)^\frac{1}{p}\|g\|_{L_p}.
\end{equation}

Additionally, by the Fundamental Theorem of Calculus, 
$$g(x_n)=g(y_n)+\int_{y_n}^{x_n}g'(t)dt.$$
Using the inequality $|a+b|^p\leq 2^{p-1}(|a|^p+|b|^p)$, we have
\begin{displaymath}
 \begin{array}{lll}
  |g(x_n)|^p & \leq & 2^{p-1}\left[|g(y_n)|^p+\left|\dint_{y_n}^{x_n}g'(t)dt\right|^p\right]\\
  \\
  & \leq & 2^{p-1}\left[|g(y_n)|^p+\left(\frac{2q}{3}\right)^\frac{p}{p'}\dint_{I_n}|g'(t)|^pdt\right].\\
 \end{array}
\end{displaymath}
The second inequality follows by H\"{o}lder's Inequality because
\begin{align*}
\left|\int_{y_n}^{x_n}g'(t)dt\right|^p&\leq\left(\int_{y_n}^{x_n}|g'(t)|dt\right)^p\leq\left(\int_{I_n}|g'(t)|dt\right)^p\\ 
 &\leq\left(\int_{I_n}|g'(t)|^pdt\right)^{p\frac{1}{p}}|I_n|^\frac{p}{p'}=\left(\frac{2q}{3}\right)^\frac{p}{p'}\int_{I_n}|g'(t)|^pdt.
 \end{align*}

Therefore by \eqref{subtheorem} and the fact that $p/p'=p-1$,
\begin{displaymath}
 \begin{array}{lll}
\zsum{n}|g(x_n)|^p & \leq & 2^{p-1}\zsum{n}|g(y_n)|^p+2^{p-1}\left(\frac{2q}{3}\right)^{p-1}\zsum{n}\dint_{I_n}|g'|^p\\
\\
& \leq & 2^{p-1}\left(\frac{2q}{3}\right)^{-1}\|g\|_{L_p}^p+2^{p-1}\left(\frac{2q}{3}\right)^{p-1}\|g'\|_{L_p}^p\\
\\
& = & 2^{p-1}\left(\frac{2q}{3}\right)^{-1}\left[\|g\|_{L_p}^p+\left(\frac{2q}{3}\right)^p\|g'\|_{L_p}^p\right].\\
 \end{array}
\end{displaymath}
Now since the function $x\mapsto|x|^\frac{1}{p}$ is concave, we have
\begin{displaymath}
 \begin{array}{lll}
  \left(\zsum{n}|g(x_n)|^p\right)^\frac{1}{p} & \leq & 2^\frac{p-1}{p}\left(\frac{2q}{3}\right)^{-\frac{1}{p}}\left(\|g\|_{L_p}+\frac{2q}{3}\|g'\|_{L_p}\right)\\
  & = & 2^\frac{1}{p'}\left(\frac{3}{2q}\right)^\frac{1}{p}\|g\|_{L_p}+2^\frac{1}{p'}\left(\frac{2q}{3}\right)^\frac{1}{p'}\|g'\|_{L_p},\\
 \end{array}
\end{displaymath}
which is \eqref{ellpestimate}.
\end{proof}

From now on, we consider a fixed Riesz-basis sequence for $L_2[-\pi,\pi]$, $X:=(x_j)_{j\in\Z}$, with $B$ given by \eqref{rbasisconstant} and $q\leq Q$ as in \eqref{separationcondition}.  Note that if $X$ is a Riesz-basis sequence for $L_2[-\pi,\pi]$, then $hX$ is a Riesz-basis sequence for $L_2\left[-\frac{\pi}{h},\frac{\pi}{h}\right]$.  Indeed, define the map
$$J_{\pi\sigma}:L_2[-\pi,\pi]\to L_2[-\sigma,\sigma]\quad\textnormal{via}\quad F(x)\mapsto\left(\frac{\pi}{\sigma}\right)^\frac{1}{2}F\left(\frac{\pi}{\sigma}x\right).$$
Note that $J_{\pi\sigma}$ is a linear isometry, and hence that $\left(J_{\pi\sigma}\left(e^{-ix_j(\cdot)}\right)\right)_{j\in\Z}$ is a Riesz basis for $L_2[-\sigma,\sigma]$ with the same basis constant.

Taking $\sigma=\pi/h$, we see that $\left(h^\frac{1}{2}e^{-ihx_j(\cdot)}\right)_{j\in\Z}$ is a Riesz basis for $L_2\left[-\frac{\pi}{h},\frac{\pi}{h}\right]$ with basis constant $B$.  This is equivalent, via \cite[p.143, Theorem 9]{yo}, to the fact that given any data sequence $(y_j)\in\ell_2$, there exists a unique bandlimited function $f\in PW_\frac{\pi}{h}$ such that $f(hx_j)=y_j$ for all $j\in\Z$. Consequently, Proposition \ref{solutioninellp} implies that if $g\in W_2^k$, then there exists a unique $F\in PW_\frac{\pi}{h}$ such that $F(hx_j)=g(hx_j),\; j\in\Z$.

We now turn toward the proof of Theorem \ref{TheoreminterpsobolevbyPW}.  We have just shown the unique existence of a bandlimited interpolant satisfying \eqref{EQpwinterpcondition}, so it remains to prove \eqref{EQbernstein} and \eqref{jackson}.  First, we demonstrate that \eqref{EQpwinterpcondition} and \eqref{EQbernstein} imply \eqref{jackson}.  Indeed, by \eqref{EQpwinterpcondition}, \eqref{EQbernstein}, and Theorem \ref{madychpotter}, we have that
$$\|g-F\|_{L_2}\leq Ch^k|g-F|_{W_2^k}\leq Ch^k\left(|g|_{W_2^k}+|F|_{W_2^k}\right)\leq Ch^k|g|_{W_2^k},$$
which is \eqref{jackson}.

To prove \eqref{EQbernstein}, we use the techniques of \cite{m}.  Define the sequence of {\em first forward divided differences} via the following formula:
\begin{equation}\label{firstforwarddifference}
 g^{[1]}(hx_j):=\dfrac{g(hx_{j+1})-g(hx_j)}{h(x_{j+1}-x_j)},\quad j\in\Z,
\end{equation}
and for $k\geq2$, the $k$-th forward divided difference is defined recursively:
\begin{equation}\label{kthdifference}
 g^{[k]}(hx_j):=\dfrac{g^{[k-1]}(hx_{j+1})-g^{[k-1]}(hx_j)}{h(x_{j+k}-x_j)}\quad j\in\Z.
\end{equation}

The following pair of lemmas combine to show \eqref{EQbernstein}.

\begin{lemma}\label{appendixlemma}
Let $h>0$, $k\in\N$, and let $X$ be a Riesz-basis sequence for $L_2[-\pi,\pi]$. Given $g\in W_2^k$, let $F\in PW_\frac{\pi}{h}$ be the unique bandlimited interpolant satisfying \eqref{EQpwinterpcondition}.  Then there exists a constant depending only on $k$ and $X$ such that 
 $$|F|_{W_2^k}\leq Ch^\frac{1}{2}\left\|\left(g^{[k]}(hx_j)\right)_j\right\|_{\ell_2}.$$
\end{lemma}

We relegate the proof of this lemma to Section \ref{SECappendix} as it is somewhat technical.

\begin{lemma}\label{LEMdivideddiffboundellp}
 Let $(x_j)_{j\in\Z}$ be a strictly increasing sequence such that  $\underset{j\in\Z}\inf\;\left(x_{j+1}-x_j\right)=: q>0.$  If $h>0$, $k\in\N$, and $g\in W_p^k(\R)$, then
 \begin{equation}\label{dividedbound}
  \left\|\left(g^{[k]}(hx_j)\right)\right\|_{\ell_p}\leq\dfrac{1}{(k-1)!}\left(\dfrac{1}{hq}\right)^\frac{1}{p}\|g^{(k)}\|_{L_p}.
 \end{equation}
\end{lemma}
\begin{proof}
We give the arguments for $k=1$ and general $k$ separately to exhibit better the idea behind the proof. If $k=1$, then by H\"{o}lder's Inequality and the fact that $p/p'=p-1$,
\begin{align*}
\zsum{j}\left|\frac{g(hx_{j+1})-g(hx_j)}{h(x_{j+1}-x_j)}\right|^p &= \zsum{j}|h(x_{j+1}-x_j)|^{-p}\left|\int_{hx_j}^{hx_{j+1}}g'(t)dt\right|^p\\
&\leq
\zsum{j}|h(x_{j+1}-x_j)|^{-p}|h(x_{j+1}-x_j)|^{p-1}\int_{hx_j}^{hx_{j+1}}|g'(t)|^pdt \\ &
\leq \frac{1}{hq}\|g'\|_{L_p}^p.\end{align*}
Thus
$$\left\|\left(g^{[1]}(hx_j)\right)\right\|_{\ell_p}\leq\left(\frac{1}{hq}\right)^\frac{1}{p}\|g'\|_{L_p}.$$

If $k\geq2$ we use the fact (see for example \cite[Corollary 3.4.2]{davis} ) that
$$g^{[k-1]}(hx_j) = \frac{g^{(k-1)}(\xi_j)}{(k-1)!}$$
for some $\xi_j\in[hx_j,hx_{j+k-1}]$.  

Therefore,
\begin{align*}
(k-1)!g^{[k]}(hx_j)&=(k-1)!\frac{g^{[k-1]}(hx_{j+1})-g^{[k-1]}(hx_j)}{h(x_{j+k}-x_j)} \\
  &= \frac{g^{(k-1)}(\xi_{j+1})-g^{(k-1)}(\xi_j)}{h(x_{j+k}-x_j)}\\ & = \frac{1}{h(x_{j+k}-x_j)}\int_{\xi_j}^{\xi_{j+1}}g^{(k)}(t)dt.\end{align*}
Consequently,
\begin{displaymath}
 \begin{array}{lll}
  \left[(k-1)!\right]^p|g^{[k]}(hx_j)|^p & \leq & |h(x_{j+k}-x_j)|^{-p}\left(\dint_{\xi_j}^{\xi_{j+1}}|g^{(k)}(t)|dt\right)^p\\
  \\
  & \leq & |h(x_{j+k}-x_j)|^{-p}\left(\dint_{hx_j}^{hx_{j+k}}|g^{(k)}(t)|dt\right)^p \\
  \\
  & \leq & |h(x_{j+k}-x_j)|^{-p+p-1}\dint_{hx_j}^{hx_{j+k}}|g^{(k)}(t)|^pdt\;.\\
 \end{array}
\end{displaymath}
Thus,
\begin{align*}
\zsum{j}|g^{[k]}(hx_j)|^p\leq\frac{1}{\left[(k-1)!\right]^p}(khq)^{-1}\zsum{j}\int_{hx_j}^{hx_{j+k}}|g^{(k)}(t)|^pdt
= \frac{1}{\left[(k-1)!\right]^p}(khq)^{-1}k\|g^{(k)}\|_{L_p}^p,\end{align*}
where the $k$ in the last term comes from the fact that the integral from $hx_j$ to $hx_{j+1}$ appears $k$ times for each $j$.
We conclude that
$$\left\|\left(g^{[k]}(hx_j)\right)_j\right\|_{\ell_p}\leq\frac{1}{(k-1)!}\left(\frac{1}{hq}\right)^\frac{1}{p}\|g^{(k)}\|_{L_p}.$$
\end{proof}

We conclude with the proof of Theorem \ref{TheoreminterpsobolevbyPW} which combines the above results.

\begin{proof}[Proof of Theorem \ref{TheoreminterpsobolevbyPW}]
 Given $g\in W_2^k$, let $F\in PW_\frac{\pi}{h}$ be the unique function satisfying \eqref{EQpwinterpcondition}.  Then by Lemmas \ref{appendixlemma} and \ref{LEMdivideddiffboundellp}, there exists a constant $C$ depending only on $k$ and $X$ such that
 $$|F|_{W_2^k}\leq Ch^\frac{1}{2}\left\|\left(F^{[k]}(hx_j)\right)_j\right\|_{\ell_2} = Ch^\frac{1}{2}\left\|\left(g^{[k]}(hx_j)\right)_j\right\|_{\ell_2}\leq C|g|_{W_2^k},$$
 which is \eqref{EQbernstein}.  As commented above, \eqref{EQpwinterpcondition} and \eqref{EQbernstein} imply \eqref{jackson}, whence Theorem \ref{TheoreminterpsobolevbyPW} follows.
\end{proof}


\section{Stability of Interpolants}\label{SECstabilityinterpbandlimited}

In this section we turn to the proof of Theorem \ref{scrIuniformbounded}, which we will prove in a series of steps reminiscent of the proofs in \cite{schsiv}.  The first issue that bears discussing is how Gaussian interpolation of $W_2^k$ functions is even possible.  It was shown in \cite{schsiv} that the Gaussian matrix associated with $\lambda>0$ and the Riesz-basis sequence $X$,
$$G:=G_{\lambda,X}:=\left(e^{-\lambda(x_i-x_j)^2}\right)_{i,j\in\Z},$$
is an invertible operator on $\ell_p$ for every $1\leq p\leq\infty$.  Consequently, for a given $g\in W_2^k$, Proposition \ref{solutioninellp} implies the existence of a Gaussian interpolant of the form \eqref{DEFscriptI}, where the sequence $(a_j)$ is given by $a = G^{-1}y$ with $y_j=g(x_j), j\in\Z.$  Moreover, invertibility of $G$ provides uniqueness of the Gaussian interpolant for a given data sequence, so if $F\in PW_\pi$ is the function from Theorem \ref{TheoreminterpsobolevbyPW} that interpolates $g$ at $X$, then $\I_\lambda(g)=\I_\lambda(F)$.

Now assume $h\in L_2[-\pi,\pi]$.  Then by definition of a Riesz-basis sequence, there exists a unique sequence $(a_j)_{j\in\Z}\in\ell_2$ such that $h(t)=\sum_{j\in\Z} a_je^{-ix_jt}, t\in[-\pi,\pi]$.  Let $H$ be the extension of $h$ to all of $\R$ given by $H(u):=\sum_{j\in\Z}a_je^{-ix_ju}$, $u\in\R$. Define the shift operator $A_\ell$ for $\ell\in\Z$ by
\begin{equation}\label{shiftopdef}
 A_{\ell}(h)(t):=H(t+2\pi\ell),\quad t\in[-\pi,\pi].
\end{equation}
From \eqref{rbasisconstant}, it follows that $(A_\ell)_{\ell\in\Z}$ is a uniformly bounded set of operators on $L_2[-\pi,\pi]$ with norm at most $B^2$.  Let $A_\ell^*$ denote the adjoint of $A_\ell$.  The following is a straightforward adaptation of \cite[Theorem 3.3]{schsiv}.

\begin{theorem}\label{schsivpsitheorem}
 Let $X$ be a Riesz-basis sequence, let $\lambda>0$ be fixed, and let $f\in PW_\pi$. Let $\psi_\lambda$ denote the restriction, to the interval $[-\pi,\pi]$, of the function
 $$\Psi_\lambda(\xi):=\sqrt{\dfrac{\lambda}{\pi}}e^\frac{\xi^2}{4\lambda}\F[\mathscr{I}_\lambda(f)](\xi).$$
 Then
 \begin{equation}\label{schsivpsitheoremeq}
  \F[f] = \F[\mathscr{I}_\lambda(f)]+\sqrt{\dfrac{\pi}{\lambda}}\zsumzero{\ell}A_\ell^*\left(e^{-\frac{(\cdot+2\pi\ell)^2}{4\lambda}}A_\ell(\psi_\lambda)\right)\quad\textnormal{on }[-\pi,\pi].
 \end{equation}

Consequently, the equation
\begin{equation}\label{derivativepsiformula}
 (i\xi)^k\F[f](\xi)=(i\xi)^k\F[\mathscr{I}_\lambda(f)](\xi)+(i\xi)^k\sqrt{\dfrac{\pi}{\lambda}}\zsumzero{\ell}A_\ell^*\left(e^{-\frac{(\cdot+2\pi\ell)^2}{4\lambda}}A_\ell(\psi_\lambda)\right)(\xi)
\end{equation}
holds for every $\xi\in[-\pi,\pi]$.
\end{theorem}

\begin{lemma}\label{Aelllemma}
 The following holds for every $\ell\in\Z$ and every $\xi\in[-\pi,\pi]$:
 \begin{equation}\label{Aelllemmaeq}
 (i\xi)^kA_\ell^*\left(e^{-\frac{(\cdot+2\pi\ell)^2}{4\lambda}}A_\ell(\psi_\lambda)\right)(\xi) = A_\ell^*\left(e^{-\frac{(\cdot+2\pi\ell)^2}{4\lambda}}A_\ell\left((i\cdot)^k\psi_\lambda\right)\right)(\xi).
 \end{equation}
\end{lemma}

\begin{proof}

 It suffices to show equality of the inner products of the above expressions with the Riesz basis elements $e_j:=e^{-ix_j(\cdot)}$. Here and elsewhere, $\bracket{\cdot,\cdot}$ denotes the usual inner product on $L_2[-\pi,\pi]$. 

\begin{displaymath}
 \begin{array}{lll}
  \bracket{(i\cdot)^kA_\ell^*\left(e^{-\frac{(\cdot+2\pi\ell)^2}{4\lambda}}A_\ell(\psi_\lambda)\right),e_j} & 
   =  & \bracket{A_\ell^*\left(e^{-\frac{(\cdot+2\pi\ell)^2}{4\lambda}}A_\ell(\psi_\lambda)\right),(-i\cdot)^ke_j}\\
  \\
   & = &\bracket{e^{-\frac{(\cdot+2\pi\ell)^2}{4\lambda}}A_\ell(\psi_\lambda), A_l\left((-i\cdot)^ke_j\right)}\\
  \\
   &= &\dpiint e^{-\frac{(\xi+2\pi\ell)^2}{4\lambda}}\psi_\lambda(\xi+2\pi\ell)(i(\xi+2\pi\ell))^k e^{ix_j(\xi+2\pi\ell)}d\xi\\
  \\
   &=&  \bracket{e^{-\frac{(\cdot+2\pi\ell)^2}{4\lambda}}A_\ell\left((i\cdot)^k\psi_\lambda\right),A_\ell(e_j)}\\
  \\
   &= & \bracket{A_\ell^*\left(e^{-\frac{(\cdot+2\pi\ell)^2}{4\lambda}}A_\ell\left((i\cdot)^k\psi_\lambda\right)\right),e_j}.\\
 \end{array}
\end{displaymath}
\end{proof}

 Lemma \ref{Aelllemma} and \eqref{derivativepsiformula} imply the following:
 
 \begin{corollary}\label{Aellcorollary}
  If $\xi\in[-\pi,\pi]$, then
  \begin{equation}\label{Aellcorollaryeq}
  (i\xi)^k\F[f](\xi)=(i\xi)^k\F[\mathscr{I}_\lambda(f)](\xi)+\sqrt{\dfrac{\pi}{\lambda}}\zsumzero{\ell}A_\ell^*\left(e^{-\frac{(\cdot+2\pi\ell)^2}{4\lambda}}A_\ell\left((i\cdot)^k\psi_\lambda\right)\right)(\xi).
  \end{equation}
 \end{corollary}

 \begin{lemma}[cf. \cite{schsiv}, Corollary 3.4]\label{schsivcor34}
  Suppose that $\lambda, f,\psi_\lambda,$ and  $B$ are as defined above.  Then
  \begin{equation}\label{scrIseminormestimation}
   \left\|(i\cdot)^k\F\left[\mathscr{I}_\lambda(f)\right]\right\|_{L_2[-\pi,\pi]} \leq\sqrt{2\pi}|f|_{W_2^k}+\sqrt{\dfrac{\pi}{\lambda}}B^4\dfrac{2e^{-\frac{\pi^2}{4\lambda}}}{1-e^{-\frac{\pi^2}{4\lambda}}}\|(i\cdot)^k\psi_\lambda\|_{L_2[-\pi,\pi]}.
  \end{equation}
 \end{lemma}

 \begin{proof}
  Note that \eqref{Aellcorollaryeq} yields the appropriate bound as long as we can show that
  $$\left\|\zsumzero{\ell}A_\ell^\ast\left(e^{-\frac{(\cdot+2\pi\ell)^2}{4\lambda}}A_\ell\left((i\cdot)^k\psi_\lambda\right)\right)\right\|_{L_2[-\pi,\pi]}\leq B^4\dfrac{2e^{-\frac{\pi^2}{4\lambda}}}{1-e^{-\frac{\pi^2}{4\lambda}}}\|(i\cdot)^k\psi_\lambda\|_{L_2[-\pi,\pi]}.$$
  This is true because the desired term is bounded above by $$B^2\zsumzero{\ell}e^{-\frac{(2|\ell|-1)^2\pi^2}{4\lambda}}\|A_\ell\left((i\cdot)^k\psi_\lambda\right)\|_{L_2[-\pi,\pi]},$$ via the triangle inequality and the facts that $\|A_\ell^\ast\|\leq B^2$ for all $\ell$, and that $e^{-\frac{(\xi+2\pi\ell)^2}{4\lambda}}\leq e^{-\frac{(2|\ell|-1)^2\pi^2}{4\lambda}}$ for $\xi\in[-\pi,\pi]$.  Similarly, since $\|A_\ell\|\leq B^2$ for all $\ell\in\Z$ and $e^{-\frac{(2|\ell|-1)^2\pi^2}{4\lambda}}\leq e^{-\frac{(2|\ell|-1)\pi^2}{4\lambda}}$, the above term is majorized by
  $$B^4\zsumzero{\ell}e^{-\frac{(2|\ell|-1)\pi^2}{4\lambda}}\|(i\cdot)^k\psi_\lambda\|_{L_2[-\pi,\pi]}\leq  B^4\dfrac{2e^{-\frac{\pi^2}{4\lambda}}}{1-e^{-\frac{\pi^2}{4\lambda}}}\|(i\cdot)^k\psi_\lambda\|_{L_2[-\pi,\pi]}.$$
  The fraction comes from summing the geometric series.
 \end{proof}

 This brings us to our final proposition before the proof of the main theorem in this section:
 \begin{proposition}\label{psiproposition}
  Suppose that $\lambda,f,\psi_\lambda$, and $B$ are as defined above.  Then
  $$\|(i\cdot)^k\psi_\lambda\|_{L_2[-\pi,\pi]}\leq\sqrt{2\lambda}e^\frac{\pi^2}{4\lambda}|f|_{W_2^k}.$$
 \end{proposition}

 \begin{proof}
  We begin by taking the inner product of both sides of \eqref{Aellcorollaryeq} with $(i\cdot)^k\psi_\lambda$, and noticing that both terms on the right hand side are nonnegative.  Indeed, it is easily checked that the operator $$T_\lambda(h):=\sum_{\ell\in\Z\setminus\{0\}}A_\ell^*\left(e^{-\frac{(\cdot+2\pi\ell)^2}{4\lambda}}A_\ell(h)\right)$$ 
  is positive in the sense that $\bracket{T_\lambda(h),h}\geq0$ for all $h\in L_2[-\pi,\pi]$.  Moreover, $$\bracket{(i\cdot)^k\F[\mathscr{I}_\lambda(f)],(i\cdot)^k\psi_\lambda} = \sqrt{\frac{\pi}{\lambda}}\int_{-\pi}^\pi|\xi|^{2k}e^{-\frac{\xi^2}{4\lambda}}|\psi_\lambda(\xi)|^2d\xi\;\geq\;0$$ by the definition of $\psi_\lambda$ in Theorem \ref{schsivpsitheorem}.  Therefore, from \eqref{Aellcorollaryeq}, we have that
$$\bracket{(i\cdot)^k\F[\mathscr{I}_\lambda(f)],(i\cdot)^k\psi_\lambda}\leq\bracket{(i\cdot)^k\F[f],(i\cdot)^k\psi_\lambda}.$$
Finally, 
\begin{displaymath}
 \begin{array}{lll}
  \sqrt{\dfrac{\pi}{\lambda}}e^{-\frac{\pi^2}{4\lambda}}\|(i\cdot)^k\psi_\lambda\|_{L_2[-\pi,\pi]}^2 & = & \sqrt{\dfrac{\pi}{\lambda}}e^{-\frac{\pi^2}{4\lambda}}\dpiint|(i\xi)^k\psi_\lambda(\xi)|^2d\xi\\
  \\
  & \leq & \dpiint \left[\sqrt{\frac{\pi}{\lambda}}e^{-\frac{\xi^2}{4\lambda}}(i\xi)^k\psi_\lambda(\xi)\right]\overline{(i\xi)^k\psi_\lambda(\xi)}d\xi\\
  \\
  & = & \bracket{(i\cdot)^k\F[\I_\lambda(f)],(i\cdot)^k\psi_\lambda}\\
  \\
  & \leq & \bracket{(i\cdot)^k\F[f],(i\cdot)^k\psi_\lambda}\\
  \\
  & \leq & \|(i\cdot)^k\F[f]\|_{L_2[-\pi,\pi]}\|(i\cdot)^k\psi_\lambda\|_{L_2[-\pi,\pi]},\\
 \end{array}
\end{displaymath}
where the last step is a consequence of the Cauchy-Schwarz inequality.  The required result follows, taking into account \eqref{parseval}.
\end{proof}

Before finishing the proof of Theorem \ref{scrIuniformbounded}, we note that in light of Theorem \ref{TheoreminterpsobolevbyPW}, we need only prove the upper bound for functions $f\in PW_\pi$.  That is, if we show that for all $f\in PW_\pi$, $|\I_\lambda(f)|_{W_2^k}\leq C|f|_{W_2^k}$, then  if $F\in PW_\pi$ is the function given by Theorem \ref{TheoreminterpsobolevbyPW} for a certain $g\in W_2^k$, we have
$$|\I_\lambda(g)|_{W_2^k}=|\I_\lambda(F)|_{W_2^k}\leq C|F|_{W_2^k}\leq C|g|_{W_2^k},$$
which is the conclusion of Theorem \ref{scrIuniformbounded}.

\begin{proof}[Proof of Theorem \ref{scrIuniformbounded}]

In light of \eqref{fourierseminorm}, we need to estimate $\left\|(i\cdot)^k\F\left[\I_\lambda(f)\right]\right\|_{L_2[-\pi,\pi]}$ and $\left\|(i\cdot)^k\F\left[\I_\lambda(f)\right]\right\|_{L_2(\R\setminus[-\pi,\pi])}$.

By Lemma \ref{schsivcor34} and Proposition \ref{psiproposition},
\begin{displaymath}
 \begin{array}{lll}
\left\|(i\cdot)^k\F\left[\I_\lambda(f)\right]\right\|_{L_2[-\pi,\pi]} & \leq & \sqrt{2\pi}|f|_{W_2^k}+\sqrt{\dfrac{\pi}{\lambda}}B^4\dfrac{2e^{-\frac{\pi^2}{4\lambda}}}{1-e^{-\frac{\pi^2}{4\lambda}}}\|(i\cdot)^k\psi_\lambda\|_{L_2[-\pi,\pi]}  \\
\\
& \leq & \sqrt{2\pi}(1+4B^4)|f|_{W_2^k}.\\
 \end{array}
\end{displaymath}
In the second inequality, we used the fact that $2/(1-e^{-\frac{\pi^2}{4\lambda}})\leq4$.

On $\R\setminus[-\pi,\pi]$, we use a periodization argument:
\begin{displaymath}
 \begin{array}{lll}
  \left\|(i\cdot)^k\F\left[\I_\lambda(f)\right]\right\|_{L_2(\R\setminus[-\pi,\pi])}^2 & = & \dfrac{\pi}{\lambda}\dint_{\R\setminus[-\pi,\pi]}e^{-\frac{\xi^2}{2\lambda}}|(i\xi)^k\Psi_\lambda(\xi)|^2d\xi\\
  \\
  & = & \dfrac{\pi}{\lambda}\zsumzero{\ell}\dint_{(2\ell-1)\pi}^{(2\ell+1)\pi}e^{-\frac{\xi^2}{2\lambda}}|(i\xi)^k\Psi_\lambda(\xi)|^2d\xi\\
  \\
  & = & \dfrac{\pi}{\lambda}\zsumzero{\ell}\dpiint  e^{-\frac{(t+2\pi\ell)^2}{2\lambda}}|A_\ell((i\cdot)^k\psi_\lambda)(t)|^2dt\\
  \\
  & \leq & \dfrac{\pi}{\lambda}\zsumzero{\ell}e^{-\frac{(2|\ell|-1)^2\pi^2}{2\lambda}}\|A_\ell((i\cdot)^k\psi_\lambda)\|^2_{L_2[-\pi,\pi]}\\
  \\
  & \leq & \dfrac{B^4\pi}{\lambda}\dfrac{2e^{-\frac{\pi^2}{2\lambda}}}{1-e^{-\frac{\pi^2}{2\lambda}}}\|(i\cdot)^k\psi_\lambda\|^2_{L_2[-\pi,\pi]}\\
  \\
 & \leq & 8\pi B^4|f|_{W_2^k}^2.\\
 \end{array}
\end{displaymath}

Consequently, we have that 
$$|\I_\lambda(f)|_{W_2^k(\R)}\leq \sqrt{2\pi(1+14B^4+16B^8)}\;|f|_{W_2^k(\R)}.$$

Uniform boundedness comes from the above inequality as well as the fact that $\|\mathscr{I}_\lambda(f)\|_{L_2}\leq C\|f\|_{L_2}$ for $f\in PW_\pi$ (use the same method of proof taking $k=0$).
\end{proof}

\section{Proofs of Corollaries \ref{Ihxbounded} and \ref{CORderivativemaintheorem}}\label{SECcorollaryproofs}

We begin with the proof of the Corollary \ref{Ihxbounded}.  Arguing along the same lines as in Theorem \ref{scrIuniformbounded}, we need only show the estimate for functions $F\in PW_\frac{\pi}{h}$.  To do so, we must explore the relationship between the seminorms of the two interpolants, and that between bandlimited functions $F$ and $F^h$.

Firstly,
\begin{equation}\label{EQffhseminormrelation}
 |F^h|_{W_2^k}=h^{k+\frac{1}{2}}|F|_{W_2^k},
\end{equation}

because
\begin{displaymath}
 \begin{array}{lll}
  |F^h|_{W_2^k}^2 & = & \dpiifrac\dint_{-\pi}^{\pi}|\xi|^{2k}|\F\left[F^h\right](\xi)|^2d\xi\\
  \\
  & = & \dpiifrac\dpiint|\xi|^{2k}|\F[F](\xi/h)|^2d\xi\\
  \\
  & = & \dpiifrac h\dint_{-\frac{\pi}{h}}^\frac{\pi}{h} h^{2k}|u|^{2k}|\F[F](u)|^2du\\
  \\
  & = & h^{2k+1}|F|_{W_2^k}^2.\\
 \end{array}
\end{displaymath}

Now we consider the seminorms of the interpolants. By \eqref{DEFIhX} and the identity $\F[g(\frac{\cdot}{h})](\xi)=h\F[g](h\xi)$,

$$\F\left[I^{hX}(F)(\cdot)\right](\xi)=\frac{1}{h}\F\left[\Ih\left(F^h\right)\left(\frac{\cdot}{h}\right)\right](\xi)=\F\left[\Ih\left(F^h\right)(\cdot)\right](h\xi).$$
Putting these together, we have the following relation between the seminorms of the interpolants:
\begin{equation}\label{EQInterpFourierRelation}
\left|I^{hX}(F)\right|_{W_2^k} = \dfrac{1}{h^{k+\frac{1}{2}}}\left|\Ih\left(F^h\right)\right|_{W_2^k},
\end{equation}
which is seen as follows.
\begin{displaymath}
\begin{array}{lll}
  \left|I^{hX}(F)\right|_{W_2^k} & = & \rootpi\left\|(i\cdot)^k\F\left[\Ih\left(F^h\right)\right](h\cdot)\right\|_{L_2}\\
  \\
  & = & \rootpi\left(\dint_\R|\xi|^{2k}\left|\F\left[\Ih\left(F^h\right)\right](h\xi)\right|^2d\xi\right)^\frac{1}{2}\\
  \\
  & = &\rootpi \left(\dfrac{1}{h}\dint_\R\left|\frac{u}{h}\right|^{2k}\left|\F\left[\Ih\left(F^h\right)\right](u)\right|^2du\right)^\frac{1}{2}\\
  \\
  & = &\rootpi \dfrac{1}{h^{k+\frac{1}{2}}}\left\|(i\cdot)^k\F\left[\Ih\left(F^h\right)\right]\right\|_{L_2}\\
  \\
  & = & \dfrac{1}{h^{k+\frac{1}{2}}}\left|\Ih\left(F^h\right)\right|_{W_2^k}.\\
 \end{array}
\end{displaymath}
 
\begin{proof}[Proof of Corollary \ref{Ihxbounded}]
Let $F\in PW_\frac{\pi}{h}$ be the function given by Theorem \ref{TheoreminterpsobolevbyPW}.  Then $I^{hX}(g)=I^{hX}(F)$, and by \eqref{EQffhseminormrelation}, \eqref{EQInterpFourierRelation} and Theorem \ref{scrIuniformbounded}, we see that
$$|I^{hX}F|_{W_2^k}= \frac{1}{h^{k+\frac{1}{2}}}|\I_{_{h^2}}^X(F^h)|_{W_2^k}\leq \frac{C}{h^{k+\frac{1}{2}}}|F^h|_{W_2^k}\leq\frac{C h^{k+\frac{1}{2}}}{h^{k+\frac{1}{2}}}|F|_{W_2^k} = C|F|_{W_2^k},$$
which, on account of \eqref{EQbernstein}, gives \eqref{Iw2kbound}.
\end{proof}

\begin{proof}[Proof of Corollary \ref{CORderivativemaintheorem}]
 Let $j<k$, $0<h\leq1$, and let $F$ be the function given by Theorem \ref{TheoreminterpsobolevbyPW}.  Then, as in the proof of Theorem \ref{maintheorem}, we see via Theorem \ref{madychpotter} that
 \begin{align*}
 |I^{hX}(g)-g|_{W_2^j}&\leq|I^{hX}(F)-F|_{W_2^j}+|F-g|_{W_2^j}\\ &\leq Ch^{k-j}|I^{hX}(F)-F|_{W_2^k}+Ch^{k-j}|F-g|_{W_2^k}=:I_1+I_2\;.\end{align*}
 By Corollary \ref{Ihxbounded} and \eqref{EQbernstein}, we have
 $$I_1\leq Ch^{k-j}|F|_{W_2^k}\leq Ch^{k-j}|g|_{W_2^k},$$
and
$$I_2\leq Ch^{k-j}|g|_{W_2^k},$$
from which the corollary follows.  We note that this does give a convergence result under the operative assumption $k>j$.
\end{proof}


\section{Proof of Lemma \ref{appendixlemma}}\label{SECappendix}

First, recall that if $X$ is a Riesz-basis sequence for $L_2[-\sigma,\sigma]$, then for any sequence $y\in\ell_2$, there exists a unique function $f\in PW_\sigma$ such that $f(x_j)=y_j,j\in\Z$.  Similar to \eqref{firstforwarddifference}, we define the \textit{$k$-th forward divided difference} of a data sequence $y$ by identifying $y_j$ as a function on the data sites $X$; that is, consider $y_j=y(x_j)$.  Then we find from \cite[Theorem 2]{m}, that if $y$ is a sequence such that $y^{[k]}\in\ell_2$, then there exists a unique function $f\in PW_\sigma^k$ such that $f(x_j)=y_j,j\in\Z$.  Moreover, there is a constant, $C$, such that $\|f^{(k)}\|_{L_2}\leq C\|y^{[k]}\|_{\ell_2}$.  But as we have been considering increasing values of $\sigma$, namely $\pi/h$ where $h\to0^+$, the dependence of this constant on $h$ must be recorded carefully.  We will come to the final conclusion that the constant is of order $\sqrt{h}$.

We make the preliminary observation that if $y\in\ell_2$, then the function $F\in PW_\sigma$ and the function $G\in PW_\sigma^k$ that satisfy the interpolation conditions are in fact the same (note that $y\in\ell_2$ implies that $y^{[k]}\in\ell_2$).  This conclusion follows from the simple fact that if $F\in PW_\sigma$, then $F\in PW_\sigma^k$ since the Paley-Wiener space is closed under differentiation, and since $F-G|_{X}=0$ where $X$ is a Riesz-basis sequence for $L_2[-\sigma,\sigma]$, $F=G$  (as commented at the end of Section \ref{SECbasic}).

Therefore, it suffices simply to determine the constant, $C$, such that $\|f^{(k)}\|_{L_2}\leq C\|y^{[k]}\|_{\ell_2}$.  We do this via some intermediate steps involving spline interpolants.  Our proofs here rely heavily on the work of Madych in \cite{m}.  Essentially, we track the constants through modified proofs of \cite[Theorems 1,2]{m}.  For the sake of self-containment, we present the proofs here with the necessary modifications.

\begin{theorem}[cf. \cite{m}, Theorem 1]\label{THMMadych1}
 Suppose $X$ is a Riesz-basis sequence for $L_2[-\pi,\pi]$ and $k\in\N$.  Then for every sequence $y$ such that $y^{[k]}\in\ell_2$, and $0<h\leq1$, there is a function $F\in PW_{\pi/2h}^k$ such that $\left(y_j-F(hx_j)\right)_j\in\ell_2$ and $|F|_{W_2^k}\leq Ch^\frac{1}{2}\|y^{[k]}\|_{\ell_2}$.  Here, the constant $C$ depends on $k$ and the Riesz-basis sequence $X$.  Moreover, 
 $$\left\|\left(y_j-F(hx_j)\right)_j\right\|_{\ell_2}\leq Ch^k\|y^{[k]}\|_{\ell_2}.$$
\end{theorem}

\begin{proof}
 
Let $q$ and $Q$ be as in \eqref{separationcondition}, and $B$ the basis constant from \eqref{rbasisconstant}. By work of Golomb \cite{golomb}, such a data sequence $y$ has a unique minimal piecewise polynomial spline extension (or interpolant), $s$, of order $2k$ such that $s(hx_j)=y_j$; further, $s$ is a polynomial of degree $2k-1$ on any interval of the form $[hx_m,hx_{m+1}]$.  Then by an estimate of de Boor \cite{db}, we have the following bound:
\begin{equation}\label{EQdeBoorBound}
 \|s^{(k)}\|_{L_2}\leq k!k^{1-\frac{1}{2}}k(2k+1)(2k-1)^{k-1}\left\|\left(\left(\frac{h(x_{j+k}-x_j)}{k}\right)^\frac{1}{2}y^{[k]}_j\right)_j\right\|_{\ell_2}.
\end{equation}
Thus we see that the right hand side is at most
$$C_k Q^\frac{1}{2}h^\frac{1}{2}\|y^{[k]}\|_{\ell_2},$$
where
$$C_k=k!(k)(2k+1)(2k-1)^{k-1}.$$

Next we construct $F\in PW_{\sigma}^k$ via $\widehat{F}(\xi)=\widehat{s}(\xi)\widehat{\phi}(\xi/\sigma)$ (we will specify $\sigma$ later), where $\widehat{\phi}$ is an infinitely differentiable function with support in $[-1,1]$, with $\widehat{\phi}(\xi)=1$ for $\xi$ in some neighborhood of the origin, and $s$ is the spline interpolant discussed above.  By definition, it is evident that $\widehat{F}$ has support in $[-\sigma,\sigma]$.  Moreover, it satisfies the bound
\begin{equation}\label{EQNFkbound}
\|F^{(k)}\|_{L_2}\leq \|\phi\|_{L_1(\R)}\|s^{(k)}\|_{L_2}, 
\end{equation}
which by \eqref{EQdeBoorBound} is at most
\begin{equation}\label{EQNF0Bound}
C_kQ^\frac{1}{2}h^\frac{1}{2}\|\phi\|_{L_1(\R)}\|y^{[k]}\|_{\ell_2}=: C_{k,Q}h^\frac{1}{2}\|y^{[k]}\|_{\ell_2}.
\end{equation}
As a consequence of $\eqref{EQNFkbound}, \eqref{EQNF0Bound}$, and the Paley-Wiener Theorem, $F^{(k)}\in PW_\sigma$, and so $F\in PW_\sigma^k$.

Now let $\widehat{\phi_\sigma}(\xi):=\widehat{\phi}(\xi/\sigma)$.  We observe that
$$\widehat{s}(\xi)-\widehat{s}(\xi)\widehat{\phi_\sigma}(\xi) = (i\xi)^k\widehat{s}(\xi)\dfrac{1-\widehat{\phi_\sigma}(\xi)}{(i\xi)^k}=:(i\xi)^k\widehat{s}(\xi)\widehat{I_{\sigma,k}}(\xi).$$
Thus,
\begin{equation}\label{EQsminusF}
s(hx_j)-F(hx_j) = s^{(k)}\ast I_{\sigma,k}(hx_j) = \int_\R s^{(k)}(hx_j-x)I_{\sigma,k}(x)dx. 
\end{equation}

Using Minkowski's Integral Inequality, \eqref{EQsminusF}, and the fact that $y_j=s(hx_j)$, we find that
\begin{displaymath}
 \begin{array}{lll}
  \left\|\left(y_j-F(hx_j)\right)_j\right\|_{\ell_2}& =& \left(\zsum{j}\left|\dint_\R s^{(k)}(hx_j-x)I_{\sigma,k}(x)dx\right|^2\right)^\frac{1}{2}\\
  \\
  & \leq & \dint_\R\left(\zsum{j}\left|s^{(k)}(hx_j-x)\right|^2|I_{\sigma,k}(x)|^2\right)^\frac{1}{2}dx\\
  \\
  & = & \dint_\R\|s^{(k)}(hx_j-x)\|_{\ell_2}|I_{\sigma,k}(x)|dx.
 \end{array}
\end{displaymath}

Our next step is to estimate $\|s^{(k)}(hx_j-x)\|_{\ell_2}$.  We will see that $$\|s^{(k)}(hx_j-x)\|_{\ell_2}\leq\left(\dfrac{ND_{k-1}}{hq}\right)^\frac{1}{2}\|s^{(k)}\|_{L_2},$$
where $N:=\lceil\frac{Q}{q}\rceil$.  To show this, we require the following:
\begin{lemma}[cf. \cite{m}, Lemma 2]\label{LEMpoly}
 
Let $\mathcal{P}_m$ denote the class of algebraic polynomials of degree at most $m$.  If $P$ is a polynomial in $\mathcal{P}_m$ which is not identically zero, then for any $-\infty<a<b<\infty$,
$$1\leq\dfrac{\underset{a\leq x\leq b}\max|P(x)|}{\left(\frac{1}{b-a}\int_a^b|P(x)|^2dx\right)^\frac{1}{2}}\leq D_m,$$
where $D_m$ is a finite constant which depends on $m$ but is independent of $a$ and $b$.
\end{lemma}

To apply this to our situation, simply notice that if $hx_j-x\in[hx_m,hx_{m+1}]$ for some $m$, then $s^{(k)}$ is a polynomial of degree at most $k-1$ on this interval, and consequently
$$|s^{(k)}(hx_j-x)|^2\leq\dfrac{D_{k-1}}{h(x_{m+1}-x_m)}\dint_{hx_m}^{hx_{m+1}}|s^{(k)}(y)|^2dy
,$$
where $D_{k-1}$ is the constant from Lemma \ref{LEMpoly}.  Since $hq\leq h(x_{m+1}-x_m)\leq hQ$ for all $m$, there are at most $N=\lceil\frac{Q}{q}\rceil$ terms of the form $hx_j-x$ in any given interval $[hx_m,hx_{m+1}]$.  Consequently, we have that for any $x\in\R$,
$$\zsum{j}|s^{(k)}(hx_j-x)|^2=\zsum{m}\left(\underset{hx_j-x\in[hx_m,hx_{m+1}]}\dsum|s^{(k)}(hx_j-x)|^2\right)\leq\dfrac{ND_{k-1}}{hq}\|s^{(k)}\|_{L_2}^2,$$
which is what we set out to show.

Finally, noting that $\|I_{\sigma,k}\|_{L_1}=\sigma^{-k}\|I_{1,k}\|_{L_1}$, and taking $\sigma=\pi/(2h)$, we obtain
$$\left\|\left(y_j-F(hx_j)\right)_j\right\|_{\ell_2}\leq C_{k,q,Q}h^k\|y^{[k]}\|_{\ell_2},$$
where
$$C_{k,q,Q} = C_{k,Q}\left(\frac{ND_{k-1}}{q}\right)^\frac{1}{2}\frac{2^k}{\pi^k}\|I_{1,k}\|_{L_1}.$$
\end{proof}

\begin{proof}[Proof of Lemma \ref{appendixlemma}]
Let $g\in W_2^k$, and let $y_j:=g(hx_j)$. Let $F_0$ be the $PW_{\pi/2h}^k$ function given by Theorem \ref{THMMadych1}.  Since $\left(y_j-F_0(hx_j)\right)_j\in\ell_2$, and $hX$ is a Riesz-basis sequence for $L_2\left[-\frac{\pi}{h},\frac{\pi}{h}\right]$, there exists a unique $F_1\in PW_\frac{\pi}{h}$ such that $F_1(hx_j)=y_j-F_0(hx_j),j\in\Z$.  Moreover, this function satisfies $\|F_1\|_{L_2}\leq \sqrt{2\pi}Bh^\frac{1}{2}\left\|\left(y_j-F_0(hx_j)\right)_j\right\|_{\ell_2}$.  Recall from the discussion in Section \ref{SECinterpbybandlimited} that $\left(h^\frac{1}{2}e^{-ihx_j(\cdot)}\right)_{j\in\Z}$ is a Riesz basis for $L_2\left[-\frac{\pi}{h},\frac{\pi}{h}\right]$ with basis constant $B$, the basis constant for $\left(e^{-ix_j(\cdot)}\right)_{j\in\Z}$.  Therefore by \eqref{rbasisconstant},
$$\|F_1\|_{L_2} \leq B\left(\zsum{j}\left|\bracket{\widehat{F_1},h^\frac{1}{2}e^{-ihx_j(\cdot)}}_{L_2\left[-\frac{\pi}{h},\frac{\pi}{h}\right]}\right|^2\right)^\frac{1}{2} = \sqrt{2\pi}Bh^\frac{1}{2}\left\|\left(F_1(hx_j)\right)_j\right\|_{\ell_2}, 
$$
the last step coming from the inversion formula.

Define $F:=F_0+F_1$.  Then by construction, $F(hx_j)=y_j=g(hx_j),j\in\Z$.  Moreover, we have that
$$\|F_1\|_{L_2}\leq\sqrt{2\pi}Bh^\frac{1}{2}\left\|\left(F_1(hx_j)\right)_j\right\|_{\ell_2}=\sqrt{2\pi}Bh^\frac{1}{2}\left\|\left(y_j-F_0(hx_j)\right)_j\right\|_{\ell_2}$$ $$\leq C_{k,q,Q}\sqrt{2\pi}Bh^{k+\frac{1}{2}}\|y^{[k]}\|_{\ell_2}.$$
Recalling that for $F\in PW_\sigma$, the relation $|F|_{W_2^k}\leq\sigma^k\|F\|_{L_2}$ holds, we see that
\begin{equation}\label{EQNF1AppendixLemma}
 |F_1|_{W_2^k}\leq C_{k,q,Q,B}h^\frac{1}{2}\|y^{[k]}\|_{\ell_2}.
\end{equation}

Therefore \eqref{EQNF0Bound} and \eqref{EQNF1AppendixLemma} lead to the conclusion that
$$|F|_{W_2^k}\leq C_{k,q,Q,B}h^k\|y^{[k]}\|_{\ell_2}+C_{k,Q}h^\frac{1}{2}\|y^{[k]}\|_{\ell_2}\leq Ch^\frac{1}{2}\|y^{[k]}\|_{\ell_2},$$
where $C$ depends on $k,q,Q,$ and $B$ (we assumed $h\leq1$, so we note that $h^\frac{1}{2}$ is the biggest term involving $h$). This concludes the proof.
\end{proof}

\section{Interpolation by Means of Regular Interpolators}\label{SECledford}

Here we extend the discussion of approximation and convergence rates of $W_2^k$ functions, but rather than limiting ourselves to the case of the Gaussian interpolation operator, we consider families of \emph{regular interpolators}, a notion developed by Ledford \cite{ledford}.

A function $\phi:\R\to\R$ is said to be an \textit{interpolator} for $PW_\pi$ if the following hold:
\begin{description}
 \item[(A1)] $\phi\in L_1(\R)\cap C(\R)$ and $\widehat{\phi}\in L_1(\R)$.
 \item[(A2)] $\widehat{\phi}(\xi)\geq0,\xi\in\R$, and $\widehat{\phi}(\xi)>0$ on $[-\pi,\pi]$.
 \item[(A3)] If $M_j:=\underset{|\xi|\leq\pi}\sup\widehat{\phi}(\xi+2\pi j)$, then $(M_j)_{j\in\Z}\in\ell_1$.
\end{description}

Next, a family of interpolators $(\phi_\alpha)_{\alpha\in A}$ indexed by an unbounded set $A\subset(0,\infty)$ is called \textit{regular} if
\begin{description}
 \item[(R1)] $\phi_\alpha$ is an interpolator for $PW_\pi$ for every $\alpha\in A$.
 \item[(R2)] Let $M_j(\alpha):=\underset{|\xi|\leq\pi}\sup\widehat{\phi_\alpha}(\xi+2\pi j),$ and $m_\alpha:=\underset{|\xi|\leq\pi}\inf\widehat{\phi_\alpha}(\xi)$. Then there exists a constant $C$ independent of $\alpha$ such that for all $\alpha\in A$, $\zsumzero{j}M_j(\alpha)\leq Cm_\alpha$.
 \item[(R3)] For almost every $\xi\in[-\pi,\pi]$, $\inflim{\alpha} \dfrac{m_\alpha}{\widehat{\phi_\alpha}(\xi)}=0$.
\end{description}
Here, $\alpha$ plays the role of $1/h$ in our previous discussions.

It was shown (\cite[Corollary 1]{ledford}) that if $X$ is a Riesz-basis sequence for $L_2[-\pi,\pi]$ and $f\in PW_\pi$, then there is a unique sequence $(a_j)_{j\in\Z}\in\ell_2$ such that the interpolant
\begin{equation}\label{DEFinterpoperator}
\I (f)(x):=\zsum{j}a_j\phi(x-x_j)
\end{equation}
is continuous and satisfies $\I (f)(x_j)=f(x_j),\;j\in\Z$.

The main result of that paper is the following:
\begin{theorem}[cf. \cite{ledford}, Theorems 1 and 2] 
 If $(\phi_\alpha)_{\alpha\in A}$ is a set of regular interpolators and $X$ a Riesz-basis sequence for $L_2[-\pi,\pi]$, then for every $f\in PW_\pi$,
 $$\inflim{\alpha} \I_\alpha (f) = f,$$
 both in $L_2(\R)$ and uniformly on $\R$.
 \end{theorem}
Here $\I_\alpha$ is defined as in \eqref{DEFinterpoperator} with $\phi$ replaced by $\phi_\alpha$.

Given a sequence of regular interpolators, and their associated interpolation operators $\I_\alpha:PW_\pi\to L_2$, we define another set of interpolation operators via
$$I^\alpha (f)(x):=\alpha\I_{\alpha^2}\left(f^\frac{1}{\alpha}\right)(\alpha x),\quad\textnormal{where}\; f^\frac{1}{\alpha}(x):=\frac{1}{\alpha}f\left(\frac{x}{\alpha}\right).$$

It is easily verified that if $f\in PW_{\alpha\pi}$, then $f^\frac{1}{\alpha}\in PW_\pi$ and $I^\alpha (f)\left(\frac{x_j}{\alpha}\right)=f\left(\frac{x_j}{\alpha}\right),\;j\in\Z.$  Indeed
$$I^\alpha (f)\left(\frac{x_j}{\alpha}\right)=\alpha\I_{\alpha^2}\left(f^\frac{1}{\alpha}\right)\left(\frac{\alpha x_j}\alpha\right)=\alpha\frac{1}{\alpha}f\left(\frac{x_j}{\alpha}\right).$$

Again, the idea here is that to interpolate Sobolev functions $g\in W_2^k$, we first interpolate them by bandlimited functions of increasing band (namely by $F\in PW_{\alpha\pi}$), where the interpolation is done at the shrinking set of points $\left(\frac{x_j}{\alpha}\right)$.  We summarize the analogous results to those found in Section \ref{SECmainresults}, and note that the proofs follow the same techniques employed in Section \ref{SECstabilityinterpbandlimited}.

\begin{theorem}\label{THMledfordmain}
 Let $(\phi_\alpha)_{\alpha\in A}$ be a set of regular interpolators for $PW_\pi$.  Let $k\in\N$ and let $X$ be a Riesz-basis sequence for $L_2[-\pi,\pi]$.  Then there exists a constant depending only on $k, X$ and the functions $(\phi_\alpha)$ such that for every $g\in W_2^k$,
 \begin{equation}\label{EQledfordmaintheorem}
  \|I^{\alpha}(g)-g\|_{L_2}\leq C\alpha^{-k}|g|_{W_2^k}.
 \end{equation}
\end{theorem}

\begin{corollary}\label{CORledfordderivativemaintheorem}
 Let $(\phi_\alpha)_{\alpha\in A}$ be a set of regular interpolators for $PW_\pi$.  Let $k\geq2$, $1\leq j<k$ and let $X$ be a Riesz-basis sequence for $L_2[-\pi,\pi]$.  Then there exists a constant depending on $j, k, X$ and the functions $(\phi_\alpha)$ such that for every $g\in W_2^k$,
 \begin{equation}\label{EQledfordderivativemaintheorem}
  |I^{\alpha}(g)-g|_{W_2^j}\leq C\alpha^{j-k}|g|_{W_2^k}\;.
 \end{equation}
\end{corollary}

\begin{corollary}\label{CORledfordschwartzresult}
 Let $(\phi_\alpha)_{\alpha\in A}$ be a set of regular interpolators for $PW_\pi$ and let $X$ be a Riesz-basis sequence for $L_2[-\pi,\pi]$.  Then for each $k\in\N$, there exists a constant depending on $X$ and the functions $(\phi_\alpha)$ such that for every $\psi\in\mathcal{S}(\R)$ and $\psi\in PW_\sigma$ for some $\sigma>0$,
 \begin{equation}\label{EQledfordschwartz}
  \|I^{\alpha}(\psi)-\psi\|_{L_2}\leq C\alpha^{-k}|\psi|_{W_2^k}.
 \end{equation}
\end{corollary}

To conclude this section, we give four examples of families of regular interpolators.  The first example is the family of Gaussians we have already discussed at length, namely
$$\mathcal{G}:=\left(e^{-\lambda x^2}\right)_{\lambda\in(0,1]}.$$
In Ledford's notation, $\alpha$ here corresponds to $1/\lambda$.  The next two examples are given in \cite{ledford}: the family of Poisson kernels,
$$\mathcal{F}:=\left(\sqrt{\dfrac{2}{\alpha}}\dfrac{\alpha}{\alpha^2+x^2}\right)_{\alpha\in[1,\infty)},$$
and the sequence of differenced convolutions of the Hardy multiquadric, $\phi(x):=\sqrt{1+x^2}$,
$$\mathcal{M}:=\left((-1)^k\Delta^k\phi_\ast^k(x)\right)_{k\in\N},$$
where $\Delta^1f(x):=f(x+1)+f(x-1)-2f(x)$, and recursively $\Delta^kf(x):=\Delta^1\left(\Delta^{k-1}f\right)(x)$, and $f_\ast^k(x):=\left(f\underbrace{\ast\dots\ast}_{k}f\right)(x)$.

Our final example is one which is not listed in \cite{ledford}, but an important one nonetheless.  Given a fixed $\beta\in(-\infty,-1/2)$, the class of inverse multiquadrics
$$\mathcal{Q}_\beta:=\left((x^2+c^2)^\beta\right)_{c\in[1,\infty)}$$
is a family of regular interpolators.

\section{Remarks on the Multivariate Case}\label{SECmultivariate}

A simple extension of the main theorem (Theorem \ref{maintheorem}) can be made by a tensor product argument. This requires considering Riesz-basis sequences for $L_2[-\pi,\pi]^d$ that form a grid (i.e. Cartesian products of $d$ Riesz-basis sequences for $L_2[-\pi,\pi]$).  However, this is far from the general case of nonuniform data sites in higher dimensions.  The reason for the lack of a better multivariate extension is twofold.  Firstly, we do not have a proper multidimensional version of Theorem \ref{TheoreminterpsobolevbyPW}.  Secondly, not much is known about Riesz-basis sequences associated with more general sets in higher dimensions.  For example, it is unknown whether or not such a sequence exists for $L_2(B_2)$, where $B_2$ is the Euclidean ball in 2 dimensions.  For a more involved discussion, consult page 525 of \cite{bss}.

\end{document}